\documentclass{amsart}
\usepackage{amsmath,amsthm,amscd}

\usepackage{amssymb}

\newtheorem{thm}{Theorem}

\newtheorem{lem}[thm]{Lemma}
\newtheorem{cor}[thm]{Corollary}

\newtheorem{prop}[thm]{Proposition}

\newtheorem{conj}[thm]{Conjecture}
   
\theoremstyle{definition}
\newtheorem{defn}[thm]{Definition}

\newtheorem{say}[thm]{}
\newtheorem{exmp}[thm]{Example}

\newtheorem{ques}[thm]{Question}    

\newtheorem{rem}[thm]{Remark}          

\newtheorem*{ack}{Acknowledgments}      

\newtheorem{defn-thm}[thm]{Definition--Theorem}  
\newtheorem{defn-lem}[thm]{Definition--Lemma}  

\theoremstyle{remark}

\renewcommand{\c}[0]{{\mathbb C}}  

\renewcommand{\o}[0]{{\mathcal O}} 
\newcommand{\z}[0]{{\mathbb Z}}

\newcommand{\p}[0]{{\mathbb P}}

\newcommand{\q}[0]{{\mathbb Q}}

\newcommand{\qtq}[1]{\quad\mbox{#1}\quad}

\newcommand{\red}[0]{\operatorname{red}}

\newcommand{\sing}[0]{\operatorname{Sing}}

\newcommand{\cl}[0]{\operatorname{Cl}}

\newcommand{\tsum}[0]{\textstyle{\sum}}

\def\into{\DOTSB\lhook\joinrel\to}

\newcommand{\CC}{{\mathbb{C}}}

\newcommand{\PP}{{\mathbb{P}}}

\newcommand{\SSS}{{\mathbb{S}}}

\begin{document}
\bibliographystyle{amsalpha}


\title{Homotopically trivial deformations}
\author{Javier Fernandez de Bobadilla}
\author{J\'anos Koll\'ar}

\maketitle

The aim of this note is to call attention to a question
about deformations that are homotopically trivial. First we need a definition.

 \begin{defn} \label{homot.fibbund.defn}
A proper morphism of complex spaces $f:X\to Y$
is called a {\it homotopy fiber bundle} if $Y$ has an open cover
$Y=\cup U_i$ such that  for every $i$ and for every $y\in U_i$ the inclusion
$$
f^{-1}(y)\into  f^{-1}(U_i)\qtq{is a homotopy equivalence.}
$$
For every $y\in Y$ there is an open neighborhood
$y\in U_y$ such that $f^{-1}(y) $ is a deformation retract
of $f^{-1}(U_y) $. 
Choose a retraction $r_y: f^{-1}(U_y)\to f^{-1}(y) $.
Thus $f^{-1}(y)\into  f^{-1}(U_y) $
is a homotopy equivalence and so $f$ is a homotopy fiber bundle
iff for every $y\in Y$ and  $y'\in U_y$ the induced map
$$
r_{y'\to y}: f^{-1}(y')\to f^{-1}(y)\qtq{is a homotopy equivalence.}
$$
Similarly, if $R$ a  commutative ring, then $f:X\to Y$
is called an {\it $R$-homology fiber bundle} if
$$
H_*\bigl(f^{-1}(y), R\bigr)\to H_*\bigl(f^{-1}(U_i), R\bigr)
\qtq{is an isomorphism.}
$$
As above, these conditions hold iff  the
retraction maps $r_{y'\to y}$ induce isomorphisms
$$
\bigl(r_{y'\to y}\bigr)_*: 
H_*\bigl(f^{-1}(y'), R\bigr)\to H_*\bigl(f^{-1}(y), R\bigr).
$$
We are mostly interested in cases when the fibers of $f$ are
irreducible. 

If the fibers are reducible,  some pathological  cases
are shown by Example \ref{exmp.nonormal.2}.
To avoid these, 
one should also assume that
the image of the fundamental class $\bigl[f^{-1}(y)\bigr] $
in $H_*\bigl(f^{-1}(U_i), R\bigr) $ 
is independent of $y$. Equivalently, the
retraction 
$r_{y'\to y}$ maps $\bigl[f^{-1}(y')\bigr] $  to  $\bigl[f^{-1}(y)\bigr] $.

All the examples of $\z$-homology fiber bundles that we know
are also homotopy  fiber bundles but being a $\q$-homology fiber bundle
is a much weaker property.
\end{defn}

The main problem we want to consider is the following.

\begin{ques} Let  $f:X\to Y$ be a homotopy or $\z$-homology fiber bundle.
Under what conditions is it a topological or differentiable fiber bundle?
\end{ques}

If $X$ is non-normal, it is easy to give examples of homotopy fiber bundles 
$f:X\to Y$ where not all fibers  $f^{-1}(y) $ are  
homeomorphic to each other, see Example~\ref{exmp.nonormal}.
In the first version of \cite{k-p} it was asked whether
 the answer was positive for normal spaces. We show in 
 Example \ref{main.counter.exmpp} that
 this is not the case. 
However, we still hope that for smooth varieties
the situation is as nice as possible.

\begin{conj}\label{smooth.conj}
 Let  $f:X\to Y$ be a homotopy or $\z$-homology fiber bundle
such that $X$ is smooth.
Then $f$ is smooth hence  $f:X\to Y$ is  a  differentiable fiber bundle.
\end{conj}

A stronger version, more closely related to deformation theory
is formulated as Conjecture \ref{smooth.conj.V}.

\begin{say}[Origin of the conjecture]\label{origin.say} 
We were led to this question
by the study of  universal covers of projective varieties.
Their modern study  was  initiated by
 Shafarevich \cite[Sec.IX.4]{shaf};
see \cite{shaf-book, bog-kat, nakayama,  ekpr, chk} and the references there
for recent results and surveys.
One aim of these investigations is to understand  projective varieties
whose universal cover is ``simple.''
There are  many  ways to define what ``simple'' should mean;
here we focus on a topological variant considered in \cite{k-p}.

{\it Question \ref{origin.say}.1.} Describe   projective varieties $X$ whose
 universal cover $\tilde X$ is 
homotopic to a finite CW complex.

This seems to be  a rather difficult problem in general,
so here we consider a series of special cases
 that seem especially important for
applications.

Let  $X$ be a smooth projective variety  and  $f:X\to Y$  a surjective
morphism.  Let $\tilde Y\to Y$ denote 
 the universal cover.
By pull-back we obtain  $\tilde f:\tilde X\to \tilde Y$.
In light of  \cite{k-p} the following seems quite plausible.

{\it Question \ref{origin.say}.2.} Assume that $\tilde Y$ is
contractible and $\tilde X$ is 
homotopic to a finite CW complex. Does this imply that
$f$ is a homotopy fiber bundle?

Conversely, if $f$ is a homotopy fiber bundle and $\tilde Y$ is 
homotopic to a finite CW complex then most likely
$\tilde X$ is 
homotopic to a finite CW complex.
Thus if Conjecture \ref{smooth.conj} is true then
we would have a rather complete understanding of 
when a variety $X$ with a nontrivial morphism $X\to Y$
has a ``simple'' universal cover.
\end{say}

\begin{say}[First properties]
If  $f:X\to Y$
is a $\q$-homology fiber bundle then all fibers
$f^{-1}(y) $ have the same dimension and the same number of
irreducible components. Thus if $X$ is normal then, 
by taking the Stein factorization, we may assume that
all fibers are irreducible.

Assume that $g:S\to C$ is an elliptic surface such that
all reduced fibers are smooth. Then $g$ is a  $\q$-homology fiber bundle
but it is a  $\z$-homology fiber bundle only if there are no multiple fibers.

We see below that there are many $\q$-homology fiber bundles
that are not $\z$-homology fiber bundles (\ref{Q-hom.exmp.1}, \ref{wahl.exmp}).

It is much harder to get nontrivial examples of  $\z$-homology fiber bundles.
For now we note two basic results.
\end{say}

\begin{prop}\label{gen.flat.prop} Let $X,Y$ be normal spaces and $f:X\to Y$
 a $\z$-homology fiber bundle. Then every fiber of $f$  is 
generically reduced and $f$ is smooth
at every smooth point of $\red f^{-1}(y) $ for every $y\in Y$.
\end{prop}

Proof. As we noted, we may assume that
all fibers are irreducible. In the terminology of 
 \cite[I.3.9--10]{rc-book}, $f$ is a well defined family of
proper algebraic cycles. Moreover, all fibers have multiplicity 1.
Thus the scheme theoretic fibers are generically reduced 
and $f$ is smooth
at every smooth point of $\red f^{-1}(y) $ for every $y\in Y$
by \cite[I.6.5]{rc-book}.\qed

\begin{cor}\label{flat.cor} Let  $f:X\to Y$ be
 a $\z$-homology fiber bundle where $X$ is smooth and $Y$ is normal.
Then $Y$ is smooth and $f$ is flat with  local complete intersection fibers.
\end{cor}

Proof.  Pick $y\in Y$ and let $x\in \red f^{-1}(y) $ be a smooth point.
Then $f$ is smooth at $x$ by Proposition \ref{gen.flat.prop}.
Since $X$ is smooth at $x$, this implies that $Y$ is smooth at $y$. 

Thus $Y$ is smooth, $f$ is equidimensional and $X$ is smooth, 
hence Cohen-Macaulay. These imply that $f$ is flat
\cite[Exrc.III.10.9]{hartsh}. \qed

\subsection*{Families of curves}{\ }

We start with two examples of families of reducible curves..

\begin{exmp}\label{exmp.nonormal}
$X\subset \p^3_{\mathbf x}\times \c_s$ is a reducible surface
and $f:X\to  \c_s$ is the coordinate projection.
The fiber $X_0$ is the projective closure  of the 3 coordinate axes in
$\c^3$ and $X_s$ is obtained by sliding the $x_3$-axis
along the $x_2$ axis. In concrete equations
$$
X:=\bigl(x_1=x_3=0\bigr)\cup \bigl(x_2=x_3=0\bigr)\cup 
\bigl(x_1=x_2-sx_0=0\bigr)\subset \p^3_{\mathbf x}\times \c_s.
$$
It is easy to check that  $f:X\to  \c_s$ is a homotopy fiber bundle,
all the fibers are reduced, the fibers
$X_s$ are isomorphic to each other for $s\neq 0$ but $X_0$ is
not homeomorphic to $X_s$ for $s\neq 0$.

It is straightforward to modify this example and obtain
an irreducible (but still non-normal) surface $S$ with a 
proper morphism $f:S\to \c$ which is a  homotopy fiber bundle
such that not all fibers are homeomorphic to each other.
\end{exmp}

\begin{exmp}\label{exmp.nonormal.2} Here again 
$X\subset \p^3_{\mathbf x}\times \c_s$ is a reducible surface
and $f:X\to  \c_s$ is the coordinate projection.
The general fiber is a line $L_t$ 
and a conic $C_t$ intersecting at a point $p\in \p^3$.
As we approach the special fiber, the conic degenerates to a pair
of lines  $L_0+L'_0$ and the line $L_0$ is also the limit of the family
$L_t$. In concrete equations
$$
X:=\bigl(x_2=x_3-tx_1=0\bigr)\cup \bigl(x_3=x_0x_2-tx_2^2=0\bigr)
\subset \p^3_{\mathbf x}\times \c_s.
$$
In this example 
 the
retraction map  induces isomorphisms
$$
\bigl(r_{y'\to y}\bigr)_*: 
H_*\bigl(L_t\cup C_t, R\bigr)\to H_*\bigl(L_0\cup L'_0, R\bigr)
$$
but the fundamental class 
  $\bigl[L_t\cup C_t\bigr] $ is  mapped to 
$2\bigl[L_0\bigr]+ \bigl[L'_0\bigr] $ which is different from 
the fundamental class of the central fiber 
 $\bigl[L_0\cup L'_0\bigr] $.
\end{exmp}

For curves the following result completes the picture.

\begin{prop} Let  $Y$ be a normal complex space and
$f:X\to Y$  a  $\z$-homology fiber bundle of relative dimension 1
with smooth general fibers.
Let $\pi:X^n\to X$ be the normalization of $X$. Then 
\begin{enumerate}
\item $\pi:X^n\to X$ is a homeomorphism and
\item $f\circ \pi:X^n\to Y$ is smooth hence  a differentiable  fiber bundle.
\end{enumerate}
\end{prop}

Proof. As we noted before, we may assume that
$f$ has irreducible fibers. 

Let us start with the case when $Y$ is a  smooth curve.
Let $B$ be a general fiber and $B_0$ any fiber.
Let  $B_0^n$ be the 
corresponding fiber in $X^n\to Y$. Note that the retraction
$r:B\to B_0$ factors through $B_0^n$. 
It is easy to see that 
$H_1(B^n_0,\z)\to H_1(B_0,\z)$ is surjective iff
$B_0^n\to B_0$ is a homeomorphism. Thus if  $X\to Y$ is a 
$\z$-homology fiber bundle then $B_0^n\to B_0$ and $X^n\to X$ are
homeomorphisms. We may thus assume that $X$ is normal.
In particular  the fibers are reduced and
$p_a(B_0)=p_a(B)=p_g(B)$ since $B$ is smooth.

Let $B'_0\to B_0$ denote the seminormalization and 
$B''_0\to B'_0$  the normalization. If $m_i$ are the multiplicities of the
points of  $B'_0$ then
$$
p_a(B_0)\geq p_a(B'_0)=p_a(B''_0)+\tsum (m_i-1)=p_g(B''_0)+\tsum (m_i-1).
$$
We can thus estimate the topological Euler characteristic as 
$$
\begin{array}{rcl}
\chi^{top}(B_0)&=&\chi^{top}(B'_0)=\chi^{top}(B''_0)-\sum (m_i-1)=
2-2p_g(B''_0)-\sum (m_i-1)\\
&\geq & 2-2p_a(B_0)+\sum (m_i-1).
\end{array}
$$
On the other hand, if $f$ is a  homotopy fiber bundle then
$$
\chi^{top}(B_0)=\chi^{top}(B)=2-2p_g(B)=2-2p_a(B_0).
$$
Comparing these two we see that $ \sum (m_i-1)=0$
and $p_a(B_0)= p_a(B'_0)$ hence $B_0\cong B'_0\cong B''_0 $.
Thus every fiber of $f$ is smooth.

This implies the general case by applying Proposition  \ref{red.to.dim.1}
to the class of all  smooth projective curves as  ${\mathcal V}$
\qed
\medskip

\subsection*{Reduction to 1-parameter families}{\ }

Here we show that  a variant of  Conjecture \ref{smooth.conj}
 can be reduced to the case when $\dim Y=1$.
To make this precise, fix a class of  projective varieties
${\mathcal V}$ and consider the following.

\begin{conj}\label{smooth.conj.V}
Let  $f:X\to Y$ be
a projective morphism of  complex spaces, $Y$  normal. Assume that 
\begin{enumerate}
\item 
there is a Zariski dense open subset
$Y^0\subset Y$ such that the fibers of $f$ over $Y^0$ are all in ${\mathcal V}$
and
\item $f$ is  a homotopy (resp.\  $\z$-homology) fiber bundle.
\end{enumerate}
Let $\pi:X^n\to X$ be the normalization of $X$. Then 
\begin{enumerate}\setcounter{enumi}{2}
\item $\pi:X^n\to X$ is a homeomorphism and
\item $f\circ \pi:X^n\to Y$ is smooth hence  a differentiable  fiber bundle.
\end{enumerate}
\end{conj}

We can now state the precise form of the dimension reduction.

\begin{prop}\label{red.to.dim.1} Fix a class  of smooth projective varieties
${\mathcal V}$ and assume that
 Conjecture \ref{smooth.conj.V} holds for ${\mathcal V}$ whenever $\dim Y=1$.

Then Conjecture \ref{smooth.conj.V} holds for ${\mathcal V}$ in general.
\end{prop}

Proof.  Let $\Delta\subset \c$ be 
the unit disc and 
 $\phi:\Delta\to Y$  any holomorphic map whose image is not contained in
$Y\setminus Y^0$. Let $X_\phi^n$ denote the normalization of 
$X\times_Y\Delta $. By assumption, $X_\phi^n\to \Delta$ is smooth hence it is the
simultaneous normalization of $\phi^*f:X\times_Y\Delta\to \Delta$.
In particular, the normalization of the fibers are all smooth
and the normalization map is a homeomorphism.

 By \cite[Thm.1]{kollar-husks-new}, there is a
monomorphism  $Y'\to Y$ such that $\phi^*f:X\times_Y\Delta\to \Delta$
has a simultaneous normalization iff $\phi$ factors through $Y'\to Y$.
Thus $Y'=Y$, the composite
  $f\circ \pi:X^n\to Y$ is smooth and $\pi:X^n\to X$ is a homeomorphism.
\qed

\subsection*{Localization}{\ }

Motivated by Proposition \ref{red.to.dim.1}, from now on we
concentrate on 1-parameter families. That is, $X$ is a 
normal analytic space and
$f:X\to \Delta$  a proper morphism with central fiber $X_0=f^{-1}(0)$.
By shrinking $\Delta$ we may assume that 
$X\setminus X^0\to \Delta^*$ is a topological fiber bundle.

We show that if $X_0$ has  isolated singularities then
$\z$-homology fiber bundles can be characterized
in terms of the Milnor fibers of the singular points
of $X_0$.
Subsequent examples show that there are global issues for
non-isolated singularities.

\begin{prop}\label{localize.prop} Let $X$ be a normal analytic space and
$f:X\to \Delta$  a proper morphism with central fiber $X_0=f^{-1}(0)$.
Assume that $X_0$ has only isolated singularities $p_i\in X_0$.
For each $i$, let $B_i$ be a small ball around $p_i$ and
set $M_{i,t}:=X_t\cap B_i$.  (If $X_t$ is smooth, this is the Milnor fiber.) 
The following are equivalent.
\begin{enumerate}
\item  For $0<|t|\ll 1$, the retraction map $r_t:X_t\to X_0$ is an $R$-homology
equivalence.
\item For $0<|t|\ll 1$ every $M_{i,t}$ is an  $R$-homology ball.
\end{enumerate}
\end{prop}

Proof.   Choose $\Delta_{\epsilon}\subset \Delta$ 
small enough so that $X_t$ meets $\partial B_i$ 
transversely for any $i$ and any $t\in\Delta_{\epsilon}$.
One can choose the retraction such that
$r_t$ induces a homeomorphism 
$$
r_t:X_t\setminus \cup_i M_{i,t}\cong X_0\setminus \cup_i M_{i,0}.
$$
Comparing the long exact homology  sequences of the pairs
$$
r_t: \bigl( X_t, \cup_i M_{i,t}\bigr)\to \bigl( X_0, \cup_i M_{i,0}\bigr)
$$
we see that $r_t:X_t\to  X_0$ is an $R$-homology
equivalence iff the restrictions
$r_{i,t}: M_{i,t}\to M_{i,0}$ are $R$-homology
equivalences. Since the $M_{i,0}$ are contractible, the latter holds
iff the $M_{i,t}$ are $R$-homology balls. 
\qed
\medskip

When the source $X$ of the mapping $f$ is smooth, the following local result for non-isolated singularities
is a corollary of the work of A'Campo on monodromy of singularities. 
We thank A'Campo for pointing this out. 

\begin{prop}\label{localize.nonisol} Let $X$ be smooth and $p\in X$ a point. Let
$f:(X,p)\to \Delta$ be a germ of analytic mapping. Let $B$ be a Milnor ball around $p$ and $D$ a Milnor disc around $f(p)$.
Set $M_{i,t}:=X_t\cap B_i$ for any $t\in D$ (for $t\neq f(p)$ this is the Milnor fiber.) 
Let $R$ be any ring. The following are equivalent.
\begin{enumerate}
\item  For $0<|t|\ll 1$, the retraction map $r_t:X_t\to X_0$ is an $R$-homology
equivalence.
\item The morphism $f$ is smooth at $p$.
\end{enumerate}
\end{prop}

Proof.  
In~\cite{AC} it is proved that under the given hypothesis the Lefschetz number of the monodromy of the Milnor fibration
equals $0$ if $f$ is not smooth at $0$ and it is obvious that it equals $1$ if $f$ is smooth. If the retraction map is a 
$R$-homology equivalence, then the Lefschetz number of the monodromy of the Milnor fibration equals $1$.
\qed
\medskip

Milnor fibers of isolated singularities have been extensively studied.
For surfaces   the following result
seems to have been known  but not explicitly stated;
see \cite{steenbrink, wah-loo} for closely related results.
The argument below was shown to us by A.~N\'emethi.

\begin{prop} Let $X$ be a normal threefold and 
$f:X\to \Delta$ a  $\z$-homology fiber bundle whose general fiber
is smooth and whose central fiber $X_0$ 
is normal. Then $f$ is smooth, $X$ is smooth
and $f$ is a differentiable  fiber bundle.
\end{prop}

Proof. Using Proposition \ref{localize.prop}, 
we need to consider the Milnor fibers of the singular points of $X_0$.

In general, let $(s\in S)$ be an isolated surface singularity and
$M$ the Milnor fiber of a smoothing.
The link $L$ of $S$ is diffeomorphic to the boundary $\partial M$ of $M$.
Let $\mu_0, \mu_+, \mu_-$ denote the number of zero (resp.\ positive, negative)
eigenvalues of the intersection form on the middle cohomology of $M$.

If $M$ is a $\q$-homology ball then these are all 0.
By \cite[2.24]{steenbrink}, $\mu_0 + \mu_+=2p_g(s\in S)$
where $p_g$ denotes the geometric genus of the singularity $(s\in S)$.
For a normal surface singularity  $p_g(s\in S)=\dim_s R^1g_*\o_{S'}$
where $g:S'\to S$ is a resolution of singularities.
Thus if $M$ is a $\q$-homology ball then $(s\in S)$ is a rational
singularity.

(If $(s\in S)$ is an isolated non-normal surface singularity, then
$p_g(s\in S)=\dim_s R^1g_*\o_{S'}-\dim \o_{\bar S}/\o_S$ where
$\bar S\to S$ is the normalization. There are many examples where
$M$ is a $\q$-homology ball yet $(\bar s\in \bar S)$ is not a rational
singularity.)

If $M$ is a  $\z$-homology ball, then
$L\sim \partial M$ is a  $\z$-homology sphere, hence
$\cl(S)\cong H^2(L,\z)$ is trivial \cite[p.240]{mumf-top}. 
Thus $S$ is rational and
$K_S$ is Cartier; this happens only if $S$ is a Du~Val singularity.
For smoothings of isolated hypersurface singularities there are
vanishing cycles. \qed

\begin{rem}\label{Q-hom.exmp.1} This  suggests that
Conjecture \ref{smooth.conj.V} may hold for
${\mathcal V}=\{\mbox{smooth surfaces}\}$, but there are
many more cases to check. We do not even know what happens 
when the special fiber has isolated (but  non-normal)  singularities.

 By contrast, there are many normal surface singularities
whose Milnor fiber is a $\q$-homology ball.
See, for instance, \cite[5.9]{wah-loo}.
\end{rem}

\begin{exmp}\label{wahl.exmp}
Let  $X^n\subset \p^N$ be a  smooth variety and 
$Y\subset X$  a hyperplane section such that $X\setminus Y\cong \c^n$
The simplest examples are smooth quadrics $\q^{n}\subset \p^{n+1}$
where $Y$ is a tangent plane; for more complicated examples 
with $\dim X=3$ 
see  \cite{fur-nak, fur93a, fur93b}.

One gets a family of $n$--folds
$f:{\mathbf X}\to \c$ whose general fibers $X_t$ are isomorphic to $X$
and whose special fiber $X_0$ is  isomorphic to the cone over $Y$
(possibly with some embedded points at the vertex).
For quadrics an explicit example is the family
$$
\bigl(x_0^2+\cdots x_{n-1}^2+tx_n^2+tx_{n+1}^2=0\bigr)
\subset \p^n_{\mathbf x}\times \c_t.
$$
Note that the rank drops by 2 at the origin.

If $n$ is odd, this is a $\q$-homology fiber bundle but the retraction  map
$$
\z\cong H_{n+1}\bigl(X_t,\z\bigr)\to  H_{n+1}\bigl(X_0,\z\bigr)\cong\z
$$
is multiplication by 2.
In  all  the other 3-fold  examples
the retraction induces 
$$
\z\cong H_4\bigl(X_t,\z\bigr)\to  H_4\bigl(X_0,\z\bigr)\cong\z
$$
which is multiplication by $\deg X>1$.
\end{exmp}

The following lemma shows that this construction never
gives interesting $\z$-homology equivalences.

\begin{prop} \label{def.to.cone.prop}
Let $X\subset \p^N$ be a smooth projective variety and
$Y=H\cap X\subset X$ a hyperplane section. Let
$C(Y)$ denote the cone over $Y$ and
$r_Y:X\to C(Y)$ the retraction. 
Assume that $X\setminus Y$ is a $\z$-homology ball and 
 $r_Y$ is a $\z$-homology equivalence. 
Then $X$ is a linear subspace.
\end{prop}

Proof. 
Let $L\in H^2\bigl(\p^N, \z\bigr)$ denote the hyperplane class.
We will prove that cap product with $L$ gives 
isomorphisms 
$$
\cap L: H_{i+2}(X,\z)\cong H_i(X,\z)\qtq{for $0\leq i\leq 2\dim Y$.}
\eqno{(\ref{def.to.cone.prop}.1)}
$$
Composing the even ones gives an isomorphism
$$
\bigl(\cap L\bigr)^{\dim X}:H_{2\dim X}(X,\z)\cong H_0(X,\z).
$$
Thus $\deg X=1$ and so $X$ is a linear subspace.

Since $r_Y$ is a $\z$-homology equivalence, (\ref{def.to.cone.prop}.1)
is equivalent to
$$
\cap L: H_{i+2}\bigl(C(Y),\z\bigr)\cong H_i\bigl(C(Y),\z\bigr)
\qtq{for $0\leq i\leq 2\dim Y$.}
\eqno{(\ref{def.to.cone.prop}.2)}
$$
This map can be factored as the Gysin map
$H_{i+2}\bigl(C(Y),\z\bigr)\to H_i\bigl(Y,\z\bigr)$
followed by the inclusion map
$H_i\bigl(Y,\z\bigr)\to H_i\bigl(C(Y),\z\bigr)$.
 
Taking the cone over a cycle gives a
natural isomorphism
$H_i(Y,\z)\cong H_{i+2}\bigl(C(Y), \z\bigr)$
and the  Gysin map is its inverse.
Again using that $r_Y$ is a $\z$-homology equivalence,
$H_i\bigl(Y,\z\bigr)\to H_i\bigl(C(Y),\z\bigr)$
is isomorphic to the inclusion map
$H_i\bigl(Y,\z\bigr)\to H_i\bigl(X,\z\bigr)$.
The latter is an isomorphism for $i\leq 2\dim Y$ 
since $X\setminus Y$ is a $\z$-homology ball.
This shows (\ref{def.to.cone.prop}.1).
 \qed

\subsection*{Families of cubic hypersurfaces}{\ }

In \cite{Bo1} several families with constant L\^e numbers and
 non-constant topology are produced. One of them is a 
family of homogeneous polynomials, giving 
 examples of homotopy fiber bundles which are not locally trivial 
topologically. All the examples in  \cite{Bo1} are non-normal
but here we construct a normal variant.
 Notice that all these examples
belong to a class of non-isolated singularities 
that has been studied systematically in \cite{BoMa}.

\begin{exmp}\label{main.counter.exmpp}
Consider the family of homogeneous cubic polynomials
\[
f_t(x_1,x_2,x_3, y_1,y_2,y_3):=(y_1,y_2,y_3)\cdot\left(
\begin{array}{ccc}  tx_1 & x_2 & x_3 \\ 
x_2 & tx_3 & x_1 \\ x_3 & x_1 & tx_2 \end{array}\right) \cdot 
\left(\begin{array}{c} y_1 \\ y_2 \\ y_3 \end{array} \right).
\] 
Set $F(t,{\mathbf x}, {\mathbf y})=f_t({\mathbf x}, {\mathbf y})$
and $C_0:=\CC_t\setminus\{0,-2,-2\xi,-2\xi^2\}$ where
  $\xi$ is a third root of unity.
Consider the family of cubic hypersurfaces
$$
X:=\bigl(F(t,{\mathbf x}, {\mathbf y})=0\bigr)
\subset \p^5_{{\mathbf x}, {\mathbf y} }\times C_0
$$
and let $\pi:X\to C_0$ be the second projection. 
For $t\in C_0$ the fiber $\pi^{-1}(t)$ is denoted by
$$
X_t=\bigl(f_t({\mathbf x}, {\mathbf y})=0\bigr)
\subset \p^5_{{\mathbf x}, {\mathbf y} }.
$$
We claim that $\pi:X\to C_0$
has the following properties.
\begin{enumerate}
\item   The singular set of $X_t$ is the 2-plane
$(y_1=y_2=y_3=0)$ for every $t\in C_0$.  Furthermore, $X_t$ is normal
and has only  canonical singularities.
\item $\pi:X\to C_0$
is a homotopy fiber bundle.
\item $\pi:X\to C_0$  is not topologically locally trivial in any 
neighborhood of  $t$ if 
 $\xi't^3-3t+2\xi'=0$ for some  third root of unity $\xi'$.
 (For example $t=1$ is one such value.)
\end{enumerate}
\end{exmp}
\begin{proof}
 The 2-plane $P:=(y_1=y_2=y_3=0)$ is clearly contained in $\sing X_t$.
If we project $X_t$ from $P$, the fibers are linear spaces.
By an explicit computation we see that $C_0$ was chosen such that
the fibers are all 2-dimensional. So $X_t\setminus P$ is a
rank 2 vector bundle over $\p^2$, hence smooth. This implies that $X_t$ is 
smooth  in codimension $1$, hence normal.

The projection shows that $X_t$ has a resolution $r_t:\bar X_t\to X_t$ where
$\bar X_t$ is a $\p^2$-bundle over $\p^2$. The exceptional divisor
$E_t\subset \bar X_t$ is a $\p^1$-bundle over $\p^2$ but the restriction of $r_t$
gives a conic bundle structure
$r_t|_{E_t}: E_t\to P$. Corresponding to the fibers of this conic bundle,
  $P=\sing X_t$ is   stratified according to the 
rank of the matrix
\[\left(
\begin{array}{ccc}  tx_1 & x_2 & x_3 \\ 
x_2 & tx_3 & x_1 \\ x_3 & x_1 & tx_2 \end{array}\right).\] 

 The third assertion follows from this and from 
the proof of \cite[Prop.7]{Bo1} almost word by word.
It is not worth to reproduce it, but  the key idea
 is that any homeomorphism between $X_s$ and $X_t$ carries the 
singular set of $X_s$ to the
singular set of $X_t$ and preserves the stratification. 
For generic $t$ the locus of non-maximal rank is a smooth 
cubic curve but  for $t=1$ it is a singular cubic curve.

For the second assertion we check, by a direct computation,
 the conditions of Lemmas~\ref{basiclemma} and \ref{basiclemma2}.
Alternatively, comparing the homology sequences of the pairs
$\bigl(\bar X_t, E_t\bigr)$ and $\bigl( X_t, P\bigr)$
shows that $\pi:X\to C_0$
is a $\z$-homology fiber bundle.
\end{proof}

We follow the ideas of~\cite{Bo1}.
Let $f_t:(\CC^n,0)\to(\CC,0)$ be a family of holomorphic function germs depending holomorphically on a parameter. 
Define $F:\CC^n\times\CC$ by $F(x,t):=f_t(x)$. Consider the projection $\pi:\CC^n\times\CC$ to the second factor.
Let $B_\epsilon$ be the closed ball of radius $\epsilon$ centered at the origin of $\CC^n$, let $\SSS_\epsilon$ be
its boundary sphere and let $D_\delta$ be the disk of 
radius $\delta$ centered at $0$. Denote the punctured disk by $D^*_\delta$.

\begin{lem}
\label{basiclemma}
Let $\epsilon$, $\delta$ and $\eta$ be radii such that for any $t\in D_\eta$ the restriction 
$$f_t:B_\epsilon\cap f_t^{-1}(D_\delta^*)\to D_\delta^*$$
is a locally trivial fibration. Then the following restrictions of the projection mapping are homotopy fiber bundles:
$$\pi:B_\epsilon\times D_\eta\cap F^{-1}({0}\times D_\eta)\to D_\eta,$$
$$\pi:\SSS_\epsilon\times D_\eta\cap F^{-1}({0}\times D_\eta)\to D_\eta.$$
\end{lem}

{\it Proof.}
The condition implies that for any $t\in D_\eta$ the inclusions of $f_t^{-1}(0)\cap B_\epsilon$ in 
$f_t^{-1}(D_\delta)\cap B_\epsilon$ and of $f_t^{-1}(0)\cap\SSS_\epsilon$ in 
$f_t^{-1}(D_\delta)\cap\SSS_\epsilon$ are homotopy equivalences. The condition also implies that for any 
$\xi\leq\eta$ the inclusions 
of $F^{-1}({0}\times D_\eta)\cap (B_\epsilon\times D_\xi)$ in 
$F^{-1}(D_\delta\times D_\eta)\cap (\SSS_\epsilon\times D_\xi)$ and of $F^{-1}({0}\times D_\eta)\cap (\SSS_\epsilon\times D_\xi)$ in 
$F^{-1}(D_\delta\times D_\eta)\cap (B_\epsilon\times D_\xi)$
are homotopy equivalences.

The condition and Ehresmann Fibration Theorem implies that the following restrictions of the projection mapping are 
differentiable locally trivial fibrations:
$$\pi:B_\epsilon\times D_\eta\cap F^{-1}(D_\delta\times D_\eta)\to D_\eta,$$
$$\pi:\SSS_\epsilon\times D_\eta\cap F^{-1}(D_\delta\times D_\eta)\to D_\eta.\qed$$

Usually one checks the condition of the previous Lemma  by showing, 
for any $t\in D_\eta$, that in the ball 
$B_\epsilon$ the function $f_t$ has no critical points outside $f_t^{-1}(0)$
 and that the fibers $f_t^{-1}(s)$ are transverse
to $\partial B_\epsilon$ for any $s\in D_\delta\setminus\{0\}$.

The Lemma above helps in the local case. From it one can
 deduce that certain projective morphisms are 
homotopy fiber bundles. Suppose that $f_t$ is a family of homogeneous
polynomials. Let $V(F)\subset \PP^{n-1}\times D_\eta$ be the family of projective varieties defined by the zeros of $F$.
Denote by $\pi$ the projection of $\PP^{n-1}\times D_\eta$ to the second factor.

\begin{lem}\label{basiclemma2}
Suppose that the condition of the previous lemma is satisfied, and in addition that $f_t$ is a family of homogeneous
polynomials. Then the restriction of the projection
$$\pi:V(F)\to D_\eta$$
is a homotopy fiber bundle. 
\end{lem}
\begin{proof}
It is enough to prove that for any $\xi\leq \eta$ the inclusion of 
\begin{equation}
\label{inclusion}
V(f_0)\hookrightarrow V(F)\cap\pi^{-1}(D_\xi)
\end{equation}
is a homotopy equivalence. 
By the previous lemma we know that 
$$\pi:\SSS_\epsilon\times D_\eta\cap F^{-1}({0}\times D_\eta)\to D_\eta$$ is a homotopy fiber bundle. Therefore 
the inclusion 
$$F^{-1}(0,0)\cap (\SSS_\epsilon\times\{0\})\hookrightarrow F^{-1}({0}\times D_\eta)\cap (\SSS_\epsilon\times D_\xi)$$
is a homotopy equivalence for any $\xi\leq\eta$. Thus the induced homomorphisms of homotopy groups are isomorphisms.

There is an free action of the sphere $\SSS^1$ of complex numbers of modulus $1$
which is equivariant with respect to the inclusion whose quotient is the inclusion~(\ref{inclusion}). 
Applying the long exact sequence of homotopy groups associated to the
fibrations given by the quotients of the free action we conclude that the inclusion~(\ref{inclusion}) induces 
isomorphisms of homotopy groups. Whitehead's Theorem implies that then it is a homotopy equivalence.
\end{proof}

\begin{rem}
The proof of \cite[Thm.10]{Bo1} gives that if $B_\epsilon$ is a Milnor ball of $f_0$ and the L\^e numbers of $f_t$ with respect
to a prepolar coordinate system (a sufficiently generic coordinate system, 
see~\cite[p.26]{Ma} for a precise definition), then the condition in Lemma~\ref{basiclemma} is satisfied.
\end{rem}

\begin{ack} We thank N.~A'Campo, B.~Claudon,  A.~H\"oring, 
 M.~Marco-Buzun\'ariz,
 A.~N\'e\-me\-thi, J.~Pardon
and  J.~Wahl for useful comments, references and suggestions.
 JFdB was partially supported by the ERC Starting Grant project 
TGASS
and by Spanish Contract MICINN2010-2170-C02-01.
Partial financial support   to JK was provided by  the NSF under grant number 
DMS-0758275.
\end{ack}

\def\cprime{$'$} \def\dbar{\leavevmode\hbox to 0pt{\hskip.2ex \accent"16\hss}d}
  \def\cprime{$'$} \def\cprime{$'$}
  \def\polhk#1{\setbox0=\hbox{#1}{\ooalign{\hidewidth
  \lower1.5ex\hbox{`}\hidewidth\crcr\unhbox0}}} \def\cprime{$'$}
  \def\cprime{$'$} \def\cprime{$'$} \def\cprime{$'$}
  \def\polhk#1{\setbox0=\hbox{#1}{\ooalign{\hidewidth
  \lower1.5ex\hbox{`}\hidewidth\crcr\unhbox0}}} \def\cdprime{$''$}
  \def\cprime{$'$} \def\cprime{$'$} \def\cprime{$'$} \def\cprime{$'$}
\providecommand{\bysame}{\leavevmode\hbox to3em{\hrulefill}\thinspace}
\providecommand{\MR}{\relax\ifhmode\unskip\space\fi MR }
\providecommand{\MRhref}[2]{%
  \href{http://www.ams.org/mathscinet-getitem?mr=#1}{#2}
}
\providecommand{\href}[2]{#2}

\bigskip

\noindent Princeton University, Princeton NJ 08544-1000, USA

\begin{verbatim}kollar@math.princeton.edu\end{verbatim}

\bigskip 

\noindent ICMAT. CSIC-Complutense-UAM-Carlos III.
C/ Nicol\'as Cabrera, n. 13-15.

Campus Cantoblanco, UAM.
28049 Madrid.

\begin{verbatim}javier@icmat.es\end{verbatim}

\end{document}